\title{\bf Exponential Objects in Categories of Generalized Uniform Hypergraphs}

\author{Martin Schmidt\\
	\small Department of Mathematics and Informatics\\[-0.8ex]
	\small Chiba University\\[-0.8ex] 
	\small Chiba, Japan
        }

\documentclass[12pt]{article}
\usepackage{amsmath,amssymb,mystyle,amsfonts,amsthm,soul,bbold, fancyhdr,quotes, graphicx, fix-cm}
\usepackage[top=2.5cm,bottom=2.5cm,left=3.5cm,right=3.5cm]{geometry}

\makeatother
\theoremstyle{plain}
\newtheorem{thm}{Theorem}[section]
\newtheorem{lemma}[thm]{Lemma}
\newtheorem{proposition}[thm]{Proposition}
\newtheorem{corollary}[thm]{Corollary}

\theoremstyle{definition}
\newtheorem{definition}[thm]{Definition}
\newtheorem{example}[thm]{Example}

\begin{document}
\newpage
\maketitle

\begin{abstract}
	We construct exponential objects in categories of generalized uniform hypergraphs and use embeddings induced by nerve-realization adjunctions to show why conventional categories of graphs and hypergraphs do not have exponential objects. 
	
\end{abstract}

\section{Introduction}

It is well-known that exponential objects do not exist in conventional categories of uniform hypergraphs. To address this obstruction, we introduce categories of $(X,M)$-graphs and reflexive $(X,M)$-graphs, which are categories of presheaves\footnote{Recall a category of presheaves is one which is equivalent to a functor category $[\C C^{op},\Set]$ for some small category $\C C$.} on two-object categories (Definition \ref{D:XMGraph}). We concretely construct exponential objects for any given category of (reflexive) $(X,M)$-graphs and use these constructions to prove that the exponential closures of categories of uniform hypergraphs are categories of $(X,M)$-graphs. 

The categories of (reflexive) $(X,M)$-graphs should be thought of as categories of generalized $k$-uniform hypergraphs where $k$ is the cardinality of $X$. The objects can be viewed as generic containers for sets of vertices and sets of arcs, where a monoid $M$ informs the type of coherence involved. 

There are two features of categories of (reflexive) $(X,M)$-graphs which distinguish them from the more conventional categories of uniform hypergraphs. The first is that the edges have incidence in multisets. The other is the presence of unfixed edges. In the case $X$ is a two-elements set, the only unfixed edges are $2$-loops, which are called bands in \cite{rBS}. We prove that unfixed edges are necessary for the construction of exponentials in conventional categories of graphs. 

As a consequence of our constructions we address the problem that the category of $k$-uniform hypergraphs (as defined in \cite{wD}) lacks connected colimits, exponentials and does not continuously embed into the category of hypergraphs. We prove there is a continuous embedding of the category of $k$-uniform hypergraphs in a category of $(X,M)$-graphs (Proposition \ref{P:HyperGSG}) which preserves any relevant categorical structures (e.g., colimits, exponentials, injectives, projectives). Therefore, working in a category of (reflexive) $(X,M)$-graphs provides a better categorical environment for constructions on uniform hypergraphs.

\section{$(X,M)$-Graphs}
We begin with a definition.

\begin{definition}\label{D:XMGraph} \mbox{}
	\begin{enumerate}
		\item Let $M$ be a monoid and $X$ a right $M$-set. 
		The theory for $(X,M)$-graphs, $\DD G_{(X,M)}$, is the category with two objects $V$ and $A$ and homsets given by
		\begin{align*}
		   & \DD G_{(X,M)}(V,A) \defeq X, \\
		  & \DD G_{(X,M)}(A,V) \defeq \empset, \\
		  & \DD G_{(X,M)}(V,V)\defeq \{\Id_V\},\\
		 & \DD G_{(X,M)}(A,A)\defeq M.
		\end{align*}
		Composition is defined as  $m\circ x=x.m$ (the right-action via $M$), $m\circ m'=m'm$ (monoid operation of $M$).
		\item Let $M$ be a monoid such that the set $\Fix(M)\defeq \setm{m'\in M}{\forall m\in M, m'm=m'}$ is non-empty. Let $X\defeq \setm{x_{m'}}{m'\in \Fix(M)}$ be the right $M$-set with right-action $x_{m}.m'\defeq x_{mm'}$ for each $m\in M$ and $x_{m'}\in X$. The theory for reflexive $(X,M)$-graphs, $\rG_{(X,M)}$ is the same as for $\DD G_{(X,M)}$ but with 
		\begin{align*}
		\rG_{(X,M)}(A,V)\defeq \{\ell\},
		\end{align*}
		and composition $\ell\circ m=\ell$, $\ell\circ x_{m'}=\Id_V$, and $x\circ \ell=x$ for each $m\in M$ and $m'\in \Fix(M)$.
	\end{enumerate}
	The category of $(X,M)$-graphs (resp. reflexive $(X,M)$-graphs) is defined to be the category of presheaves $\widehat{\DD G}_{(X,M)}\defeq [\DD G_{(X,M)}^{op},\Set]$ (resp. $\widehat{\rG}_{(X,M)}\defeq [\rG_{(X,M)}^{op},\Set].$) 
\end{definition}

By definition, an $(X,M)$-graph $G\colon \DD G^{op}_{(X,M)}\to \Set$ has a set of vertices $G(V)$ and a set of arcs $G(A)$ along with right-actions for each morphism in $\DD G_{(X,M)}$. For example, $x\colon V\to A$ in $\DD G_{(X,M)}$ yields a set map $G(x)\colon G(A)\to G(V)$ which takes an arc $\alpha\in G(A)$ to $\alpha.x\defeq G(x)(\alpha)$ which we think of as its $x$-incidence.\footnote{Note that we use the categorical notation of evaluation of a presheaf as a functor for the set of vertices $G(V)$ and set of arcs $G(A)$ rather than the conventional graph theoretic $V(G)$ and $E(G)$ for the vertex set and edge set.} For an element $m$ in the monoid $M$, the corresponding morphism $m\colon A\to A$ in $\DD G_{(X,M)}$ yields a right-action $\alpha.m\defeq G(m)(\alpha)$ which we think of as the $m$-associated partner of $\alpha$.  If $G$ is a reflexive $(X,M)$-graph, the $\ell$-action can be thought of as the extraction of a loop from a vertex. We call a loop equal to $x.\ell$ a distinguished loop for vertex $x$. It can be thought of as the arc-proxy for the vertex. This will allow us to map arcs to vertices, or more precisely, arcs to distinguished loops.

Each $(X,M)$-graph $G$ induces a set map $\partial_G\colon G(A)\to G(V)^X$ such that $\partial_G(\alpha)\colon X\to G(V)$ is the parametrized incidence of $\alpha$, i.e., $\alpha.x=\partial_G(\alpha)(x)$. The $x$-incidence can be recovered from a parametrized incidence by precomposition of the map $\named x\colon 1\to X$ which names the element $x$ in $X$. Observe that the $m$-associated partner of an arc $\alpha$ in $G$ has the parametrized incidence such that the following commutes
\[
\xymatrix@C=3em{ X \ar@/_1em/[rrr]_{\partial_{G}(\alpha.m)} \ar[r]^-{\langle \Id_X, \named m \rangle} & X\x M \ar[r]^-{\text{\tiny action}} & X \ar[r]^-{\partial_{G}(\alpha)} & G(V).}
\]
If $G$ is a reflexive graph, $\partial_{G}(x.\ell)=\named x\circ !_X$  where $!_X$ is the terminal set map. 

Let $X$ be a set, we define the following submonoids of the endomap monoid $\End(X)$:
\begin{align*}
& \oX\defeq \{\Id_X\}\\
& \sX \defeq \Aut(X) \quad \text{(the submonoid of automaps)}\\
& \rX \defeq \setm{f\in \End(X)}{f=\Id_X\text{ or } \exists x\in X,\ \forall x'\in X,\ f(x')=x}\\
& \srX \defeq \rX\cup \sX\\
& \hX=\hrX\defeq \End(X).
\end{align*}
Thus there is the following inclusions as submonoids in $\End(X)$
\[
\xymatrix@R=1em@C=1.25em{ \oX \ar@{>->}[d] \ar@{>->}[r]  & \sX \ar@{>->}[r] \ar@{>->}[d] & \hX \ar@{=}[d] \\ \rX \ar@{>->}[r] & \srX \ar@{>->}[r] & \hrX }
\]
The right-action of $M\subseteq \End(X)$ on $X$ is given by evaluation, e.g. $x.f\defeq f(x)$.\footnote{Note that the monoid operation on $\End(X)$ is given by $f\cdot g=g\circ f$.} 

\begin{definition}  Let  $X$ be a set. 
	\begin{enumerate}
		\item The theory for oriented $X$-graphs (resp. symmetric $X$-graphs, hereditary $X$-graphs) is defined as $\oG_X\defeq \DD G_{(X,\oX)}$ (resp. $\sG_X\defeq \DD G_{(X,\sX)}$, $\hG_X\defeq \DD G_{(X,\hX)}$). The category of oriented $X$-graphs (resp. symmetric $X$-graphs, hereditary $X$-graphs) is its category of presheaves $\widehat{\oG}_{X}$ (resp. $\widehat{\sG}_{X}$, $\widehat{\hG}_{X}$). 
		\item The theory for reflexive oriented $X$-graphs (resp. reflexive symmetric $X$-graphs, reflexive hereditary $X$-graphs) is defined as $\rG_X\defeq \rG_{(X,\rX)}$ (resp. $\srG_X\defeq \rG_{(X,\srX)}$, $\hrG_X\defeq \rG_{(X,\hrX)}$).\footnote{In each case, $X$ can be verified to be the submonoid of fixed elements given in the definition of a reflexive theory.} The category of reflexive oriented $X$-graphs (resp. reflexive symmetric $X$-graphs, reflexive hereditary $X$-graphs) is its category of presheaves $\widehat{\roG}_{X}$ (resp. $\widehat{\srG}_{X}$, $\widehat{\hrG}_{X}$). 
	\end{enumerate}
\end{definition}
The various categories of $X$-graphs can be thought of as models for $k$-uniform hypergraphs where $k$ is the cardinality of $X$ and the arcs take its incidence relation in multisets of vertices.  
\begin{example} \label{E:XGraphs}\mbox{}
	\begin{enumerate}
		\item When $X=\empset$, the categories of oriented, symmetric and hereditary $X$-graphs is the category $\Set\x \Set$. 
		\item \label{E:Bouq} When $X=1$ is a one element set, the categories of oriented, symmetric and hereditary graphs is the category of bouquets, i.e., the category of presheaves on $\xymatrix@C=1em{V \ar[r]^s & A}$ (\cite{mR}, p 18). The categories of reflexive, reflexive symmetric and reflexive hereditary $X$-graphs is the category of set retractions.
		\item  When $X=\{s,t\}$, the categories of oriented, reflexive, symmetric, reflexive symmetric graphs are the categories of directed graphs, directed graphs with degenerate edges, undirected graphs with involution in \cite{rBS}. 
		
		The following is an example of a reflexive symmetric $X$-graph where $i\colon X\to X$ denotes the non-trivial automap.	
		\begin{multicols}{2}
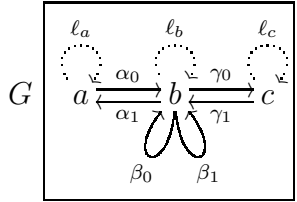
			\begin{center}
				$G\ $\framebox{ 
					$\xymatrix{ a \ar@{..>}@(ul,ur)[]^{\ell_a} \ar@<+.2em>[r]^{\alpha_0} \ar@<-.2em>@{<-}[r]_{\alpha_1} & b \ar@(d,dl)[]^{\beta_0} \ar@(d,dr)[]_{\beta_1} \ar@{..>}@(ul,ur)[]^{\ell_b} \ar@<+.2em>[r]^{\gamma_0} \ar@<-.2em>@{<-}[r]_{\gamma_1} & c \ar@{..>}@(ul,ur)[]^{\ell_c}}$}
			\end{center}
			\begin{align*}
			& G(A)= \{\alpha_0,\alpha_1,\beta_0,\beta_1,\gamma_0,\gamma_1,\ell_a,\ell_b,\ell_c \},\\
			& G(V)= \{a,\ b,\ c \}\\
			& \alpha_0.s= a,\  \alpha_0.t= b,\ \beta_0.s= b,\ \beta_0.t= b,\\
			& \gamma_0.s= c,\ \gamma_1.t= b,\\
			& a.\ell= \ell_a,\ b.\ell= \ell_b,\ c.\ell= \ell_c,\\
			& \alpha_0.i= \alpha_1,\ \beta_0.i= \beta_1,\ \gamma_0.i= \gamma_1	
			\end{align*}
		\end{multicols}
	Each loop extracted from a vertex via $\ell$ is depicted by a dotted arrow. We will call these arrows distinguished loops. They should be thought of as proxies for the vertices. Notice that for a distinguished loop $\ell_a$, we have $\ell_a.i=\ell_a$ since $\ell\circ i=\ell$ in $\srG_{X}$. However, a  non-distinguished loop may not be fixed by the right-action of $i$, as is the case with $\beta_0$ and $\beta_1$ above. If a loop $\delta$ has a distinct $i$-pair (i.e., $\delta.i\neq \delta$), we call it a nonfixed loop (or a 2-loop in the case $X=2$).\footnote{In \cite{rBS}, it is called a band.} If $\delta$ is fixed by the $i$-action (i.e., $\delta.i=\delta$) it is called a fixed loop (or a 1-loop).  
		
		To connect this definition to undirected graphs, we identify edges which are $i$-pairs and define the set of edges $G(E)$ as the quotient of the set of arrows $G(A)$ under this automorphism defined by the $i$-action.\footnote{In the subsequent, we reserve the term edge for the equivalence class of arcs under the group $\sX$.} There is an incidence operator $\partial\colon G(E)\to G(V)^2$ which defines for an $i$-pair the set of boundaries. Then an undirected representation for $G$ can be given as
\end{enumerate}
		\begin{multicols}{2}
			\begin{center}$G\ $\framebox{ 
					$\xymatrix{ a \ar@{..}@(ul,ur)[]^{\ell_a}  \ar@{-}[r]^{\alpha_0\sim \alpha_1} & b \ar@(dl,dr)@{-}[]_{\beta_0\sim \beta_1}^2  \ar@{..}@(ul,ur)[]^{\ell_b} \ar@{-}[r]^{\gamma_0\sim \gamma_1} & c \ar@{..}@(ul,ur)[]^{\ell_c}}$
				}
			\end{center}
			\begin{align*}
			& \text{ \footnotesize $G(E)= \{\alpha_0\sim\alpha_1, \beta_0\sim\beta_1,\gamma_0\sim\gamma_1,\ell_a,\ell_b,\ell_c \}$, }\\
			& \text{ \footnotesize $G(V)= \{a,\ b,\ c \}$ }\\
			& \text{ \footnotesize $(\alpha_0\sim \alpha_1).\partial=\{a,b \},\ (\beta_0\sim\beta_1).\partial=\{b,b\},$ }\\
			& \text{ \footnotesize $(\gamma_0\sim \gamma_1).\partial= \{b,c\},\ \ell_a.\partial=\{a,a \}, \ \ell_b.\partial=\{b,b\},$}\\ 
			&\text{\footnotesize $\ell.c=\{c,c\}$.}	
			\end{align*}
		\end{multicols}
\begin{enumerate}[resume]	
\item[]		We have placed a $2$ in the loop which came from the 2-loop $\beta_0\sim \beta_1$ even though the quotient has identified them. Keeping a distinction between fixed loops and nonfixed loops is necessary for constructions of exponentials (see Corollary \ref{C:Exponentials} below).\footnote{In the subsequent, if a loop has no number written inside it is assumed to be a fixed loop.}
	\end{enumerate}
\end{example}

\section{The Yoneda Embedding}
In the category of $(X,M)$-graphs the representable $\underline V\defeq \DD G_{(X,M)}(V,-)$ consists of one vertex corresponding to the identity morphism and an empty arc set. In the reflexive case, $\underline V\defeq \rG_{(X,M)}(V,-)$ also has one distinguished loop corresponding to the morphism $\ell\colon A\to V$.  The representables $\underline A\defeq \DD G_{(X,M)}(A,-)$ and $\underline A\defeq \rG_{(X,M)}(A,-)$ each have vertex set equal to $X$ corresponding to each morphism $x\colon V\to A$ and arc set equal to $M$. The right-actions are given by Yoneda, e.g., $\underline \sigma=Y(\sigma)\colon \underline A\to \underline A$. Observe that each representable has no nonfixed loops.

\begin{example} 
	 Let $X=\{s,t\}$. The Yoneda embedding gives the following diagrams, 
	
	\begin{center}
		$\text{$\widehat{\oG}_X:$}\qquad \underline V\ $\framebox{ 
			$\xymatrix{ v_1 }$} $\xymatrix{ \ar@<+.3em>[rr]^{\underline s} \ar@<-.3em>[rr]_{\underline t} &&}$
		\framebox{ 
			$\xymatrix{ v_s  \ar[rr]^{a_{1}} &&  v_t }$} $\xymatrix{ \underline A} \qquad $
	\end{center}
	\begin{center}
		$\text{$\widehat{\roG}_X:$}\qquad \underline V\ $\framebox{ 
			$\xymatrix{ v_1 \ar@{..>}@(ul,ur)[]^{a_\ell}}$} $\xymatrix{ \ar@<+1.5em>[rr]^{\underline s} \ar[rr]^{\underline t} \ar@{<-}@<-.5em>[rr]_{\underline \ell} &&}$
		\framebox{ 
			$\xymatrix{ v_s \ar@{..>}@(ul,ur)[]^{a_{s\ell}} \ar[rr]^{a_{1}} &&  v_t \ar@{..>}@(ul,ur)[]^{a_{t\ell}}}$} $\xymatrix{ \underline A \ar@(ul,ur)[]^{\underline{s}\circ \underline \ell} \ar@(dl,dr)[]_{\underline{t}\circ \underline \ell }}\qquad$
	\end{center}
	\begin{center}
		$\text{$\widehat{\sG}_X:$}\qquad \underline V\ $\framebox{ 
			$\xymatrix{ v_1 }$} $\xymatrix{ \ar@<+.3em>[rr]^{\underline s} \ar@<-.3em>[rr]_{\underline t} &&}$
		\framebox{ 
			$\xymatrix{ v_s  \ar@{-}[rr]^{a_{1}\sim a_{i}} &&  v_t }$} $\xymatrix{ \underline A \ar@(ur,dr)[]^{\underline i}}$
	\end{center}
	\begin{center}
		$\text{$\widehat{\srG}_X:$}\qquad \underline V\ $\framebox{ 
			$\xymatrix{ v_1 \ar@{..}@(ul,ur)[]^{a_\ell}}$} $\xymatrix{ \ar@<+1.5em>[rr]^{\underline s} \ar[rr]^{\underline t} \ar@{<-}@<-.5em>[rr]_{\underline \ell} &&}$
		\framebox{ 
			$\xymatrix{ v_s \ar@{..}@(ul,ur)[]^{a_{s\ell}} \ar@{-}[rr]^{a_{1}\sim a_{i}} &&  v_t \ar@{..}@(ul,ur)[]^{a_{t\ell}}}$} $\xymatrix{ \underline A \ar@(ur,dr)[]^{\underline i} \ar@(ul,ur)[]^{\underline{s}\circ \underline \ell} \ar@(dl,dr)[]_{\underline{t}\circ \underline \ell }}$
	\end{center}
	where $i\colon \{s,t\}\to \{s,t\}$ is the non-trivial automapping, $\underline s, \underline t$ are the symmetric $X$-graph morphisms which pick out $v_s$ and $v_t$ respectively, and $\underline i$ is the symmetric $X$-graph morphism which swaps $v_s$ with $v_t$ and $a_1$ with $a_i$. In the reflexive case, $\underline \ell$ is the terminal morphism, $\underline s\circ \ell, \underline t\circ \ell$ takes each arc to $a_{s\ell}$ and $a_{t\ell}$ respectively, and $\underline i$ swaps loops $a_{s\ell}$ with $a_{t\ell}$ and $a_1$ with $a_i$.
	
\end{example}

\section{Exponentials}\label{S:Exponential}
Let $G$ and $H$ be (reflexive) $(X,M)$-graphs. By Yoneda and the exponential adjunction, 
\begin{align*}
G^H(V)&=\widehat{\DD G}_{(X,M)}(\underline V,G^H)\iso \widehat{\DD G}_{(X,M)}(\underline V\x H, G)\\
G^H(A)&=\widehat{\DD G}_{(X,M)}(\underline A, G^H)\iso \widehat{\DD G}_{(X,M)}(\underline A\x H, G)
\end{align*}
with right-actions being defined by precomposition. For instance, given $f\colon \underline A\x H\to G$ (i.e., an arc in $G^H$), for each $x\in X$, $f.x\defeq f\circ (\underline x\x H)\colon \underline V\x H\to G$. The evaluation morphism is defined on components 
\begin{align*}
\ev_V&\colon \widehat{\DD G}_{(X,M)}(\underline V\x H, G)\x H(V) \to G(V), \qquad (\gamma,v)\mapsto \gamma_{V}(\Id_V,v),\\
\ev_A&\colon \widehat{\DD G}_{(X,M)}(\underline A\x H,G)\x H(A)\to G(A),\ \qquad (\delta,a)\mapsto \delta_{A}(\Id_A,a).
\end{align*}
Thus for $(X,M)$-graphs, the vertex set $G^H$ is given by $G(V)^{H(V)}$ since $\underline V$ has just a single vertex with no arcs. For reflexive $(X,M)$-graphs, since $\underline V$ is the terminal object, $\underline V\x H\iso H$, the vertex set is given by the homset $G^H(V)=\widehat{\DD G}_{(X,M)}(H,G)$.

To give a description of the arc set of the (reflexive) $(X,M)$-graph $G^H$. We define a set map analogous to taking a homset of a category
\[
\overline G\colon G(V)^X\to 2^{G(A)},\quad (v_x)_{x\in X}\mapsto \setm{\beta\in G(A)}{\forall x\in X,\ \beta.x=v_x}.
\] 
We recall that the graph $\underline A\x H$ has a parametrized incidence operator $\partial\colon \underline A\x H(A)\to \\((\underline A \x H)(V))^X$.  For each set map $f\colon X\x H(V)=(\underline A\x H)(V)\to G(V)$ we compose to obtain the following diagram.
\[
\xymatrix{ \underline A\x H(A) \ar[d]_{\partial} \ar[r]^-{G_f\defeq\overline Gf^X\partial} & 2^{G(A)} \\ (X\x H(V))^X \ar[r]^-{f^X} & G(V)^X \ar[u]_{\overline G} } 
\]
We see that $G_f\defeq \overline Gf^X\partial(a_\sigma, \alpha)$ is the set of arcs in $G$ with the same set of incident vertices determined by the value of $f$ on the incident vertices of the arc $(a_\sigma, \alpha)$ in $\underline A\x H(A)$. Observe that a morphism $g\colon \underline A\x H\to G$ is determined on the arcs of $H(A)$, i.e., given an arc $(a_\sigma, \alpha)\in \underline A\x H(A)$ we have $g_A(a_\sigma,\alpha)=g_A(a_1,\alpha).\sigma$.

The general formula for the arc set of exponentials of  non-reflexive $(X,M)$-graphs is as follows
\begin{align*}
G^H(A)& \textstyle = \bigsqcup_{f\in (G^H(V))^X}\prod_{\alpha\in H(A)}G_f(a_1,\alpha)
\end{align*}
Thus an arc in $G^H$ is given by a pair $(f=(f_x)_{x\in X},g)$ where $\left(f_x\colon H(V)\to G(V)\right)_{x\in X}$ is a family of set map and $g\colon H(A)\to G(A)$ is an element in the product $\prod_{\alpha\in H(A)}G_f(a_1,\alpha)$. Note that $((f_x)_{x\in X},g)$ is an arc in $G^H$ implies $f\colon X\x H(V)\to G(V)$ has at least one extension to a morphism $\underline A\x H\to G$. 

Given a family of set maps $\left(f_x\colon H(V)\to G(V)\right)_{x\in X}$, we define $\overline f\colon H(V)^X\to G(V)^X$ where $\overline f(h)(x)\defeq f_x(h(x))$ for each $h\in H(V)^X$ and $x\in X$. Then the set of arcs has an equivalent description 
\[
	G^H(A)=\setm{((f_x)_{x\in X},g)\in (G^H(V))^X\x G(A)^{H(A)}}{\overline f\circ \partial_H=\partial_G\circ g}
\]
i.e., it is the set of pairs $((f_x)_{x\in X},g)$ such that $g(\alpha).x=f_x(\alpha.x)$ for each $\alpha\in H(A)$ and $x\in X$. In diagram form we require that the following commute
\[
	\xymatrix{ H(A) \ar[d]_{\partial_H} \ar[r]^g & G(A) \ar[d]^{\partial_G} \\ H(V)^X \ar[r]^{\overline f} & G(V)^X.}
\]
The right-actions are given by $((f_x)_{x\in X},g).x=f_x$ for each $x\in X$ and  $((f_x)_{x\in X},g).\sigma=((f_{\sigma(x)})_{x\in X},g.\sigma)$ for each $\sigma\in M$ where $g.\sigma\colon H(A)\to G(A)$ takes $\alpha$ to $g(\alpha.\sigma)$. In other words, the following commute for each $x\in X$ and $\sigma\in M$.
\[
	\xymatrix@C=4em{ & G(V) \ar@{>->}[dr] & \\ H(V) \ar[ur]^{((f_x)_{x\in X},g).x} \ar@{>->}[r]^-{\iota_x} & \underline A\x H \ar[r]^{((f_x)_{x\in X},g)} & G} \qquad
	\xymatrix@C=4em{ \\ \underline A\x H \ar@/^3em/[rr]^-{((f_x)_{x\in X},g).\sigma} \ar[r]^-{\underline \sigma\x 1} & \underline A \x H \ar[r]^{((f_x)_{x\in X},g)} & G }
\]
where $\iota_x\colon H(V)\to \underline A\x H$ sends vertex $v$ to $(x,v)$.

In the reflexive case, given a family of morphisms $(f_x\colon H\to G)_{x\in X}$, we define $f\colon X\x H(V)\to G(V)$, $(x,v)\mapsto f_x(v)$.   Then the formula above hold for the reflexive case as well. We have 
\begin{align*}
	G^H(V)&=\widehat{\rG}_{(X,M)}(H,G)\\
	G^H(A)&=\textstyle \bigsqcup_{f\in (G^H(V))^X}\prod_{\alpha\in H(A)}G_f(a_1,\alpha).
\end{align*}
Alternatively, $G^H(A)=\setm{((f_x)_{x\in X},g)\in (G^H(V))^X\x G(A)^{H(A)} }{\overline f\circ \partial_H=\partial_G\circ g}$ as above. Then an arc in $G^H$ is given by a pair $((f_x)_{x\in X},g)$ where $(f_x)_{x\in X}$ is a family of graph morphisms $f_x\colon H\to G$ and $g\colon H(A)\to G(A)$ is an element in the product $\prod_{\alpha\in H(A)}G_f(a_1,\alpha)$. Then for each $x\in X$, $((f_x)_{x\in X},g).x=f_x$. Given a morphism $k\colon H\to G$ (i.e., a vertex in $G^H$), $k.\ell=((k)_{x\in X}, k_A)$ where $k_A\colon H(A)\to G(A)$ is the evaluation of $k$ at the arc component. For each $\sigma\in M$,  $((f_x)_{x\in X},g).\sigma=((f_{x.\sigma})_{x\in X}, g.\sigma)$ where $g.\sigma\colon H(A)\to G(A)$ takes $\alpha$ to $g(\alpha.\sigma)$.

The evaluation morphism $\ev\colon G^H\x H\to G$ for (reflexive) $(X,M)$-graphs is given as 
\begin{align*}
&\ev_V\colon G^H(V)\x H(V) \to G(V),\qquad (h, v) \mapsto h(v), \\
&\ev_A\colon G^H(A)\x H(A)\to G(A),\qquad (((f_x)_{x\in X},g),\alpha)\mapsto g(\alpha)
\end{align*}

\begin{example}\label{E:Exponentials}\mbox{}
	\begin{enumerate}
		\item \label{E:HypergraphFail}Let $X$ be a nonempty $M$-set. Then the exponential of $\underline V^{\underline V}$ in $\widehat{\DD G}_{(X,M)}$ is the terminal object $1$, which has one vertex and one fixed loop. This is an example of a creation of an arc from two $(X,M)$-graphs with no arcs. 
		
		More generally, for an arbitrary $(X,M)$-graph $G$, $G^{\underline V}$ is the $(X,M)$-graph with vertex set $G^{\underline V}=G(V)$ and arc set $G^{\underline V}(A)=G(V)^X$ with right-actions $f.x=f(x)$, $f.\sigma=f\circ \sigma$ for each $x\in X$, $f\in G^{\underline V}(A)$ and $\sigma\in M$.

		\item \label{E:XLoops} Let $X$ be a set with cardinality greater than 1 and consider the symmetric $X$-graph $\Gamma$ such that $L(V)\defeq \{v\}$ and $L(A)\defeq \{0, 1\}$ where $0.\sigma=0$ and $1.\sigma=1$ for each $\sigma\in \sX$. The vertex set for $L^{\underline A}$ is a singleton $\{v \}$ since $L(V)$ is a singleton. The set $\underline A_A=\setm{(a_1,a_\sigma)}{\sigma\in \sX}\iso \sX$ and thus the set of arcs is  $L^{\underline A}(A)\iso \Set(\sX, \{0,1\})$. We show that $L^{\underline A}$ contains a nonfixed loop. Consider a loop given by a set map $g\colon \sX\to \{0,1\}$ such that $\Id_X\mapsto 0$ and $\sigma\mapsto 1$ for the permutation $\sigma\colon X\to X$ which swaps distinct elements $x$ and $x'$ and leaves the rest fixed. Then $g.{\sigma}(\sigma)=g(
		\sigma\circ \sigma)=g(\Id_X)=0$ and thus $g.\sigma\neq g$.  Therefore $L^{\underline A}$ contains a nonfixed loop. 
		
		For example, when $X=\{s,t\}$, the exponential $L^{\underline A}$ has the following undirected representation.
\end{enumerate}
		\begin{multicols}{2}
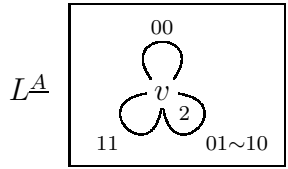
			\begin{center}
				$L^{\underline A}\ \ $\framebox{ 
					$\xymatrix{ v \ar@{-}@(ul,ur)[]^{00} \ar@{-}@(l,d)[]_{11} \ar@{-}@(d,r)[]_{01\sim 10}^2}$}
			\end{center}
			\begin{align*}
			& \text{ \footnotesize $L^{\underline A}(A)=\Set(s(2),\{0,1\})=\{00,11,01\sim 10\}$}\\
			& \text{ \footnotesize $L^{\underline A}(V)=\{v \}$}\\
			& \text{ \footnotesize $00.i=00, \ 11.i=11, 01.i=10$}
			\end{align*}
		\end{multicols}
\begin{enumerate}[resume]
\item[]		where $i\colon 2\to 2$ is the non-identity automorphism and $xy\colon s(2)\to s(2)$ is the set map $xy(\Id_X)=x,\ xy(i)=y$ for $x,y\in \{0,1\}$. Evaluation on arcs is given by projection, e.g., $\ev_A(xy,a_1)=x$. 
		\item \label{E:1Loops}  Let $X$ be a set of cardinality greater than $1$ and consider the reflexive symmetric $X$-graph $L$ such that $L(V)\defeq \{v\}$ and $L(A)\defeq \{0,1 \}$. For each $\sigma\in \sX$, we set $i.\sigma=i$ for $i=0,1$. We also set $v.\ell=0$. The vertex set of the exponential $L^{\underline A}$ is $\widehat{\rG}_{(X,\srX)}(\underline A, L)=L(A)=\{0,1 \}$ by Yoneda. Using the construction above we obtain the arrow set
		\[
		L^{\underline A}(A)=L(A)^X\x \Set(\srX,\{0,1\}).
		\]
		We show that $L^{\underline A}$ contains a nonfixed loop. Consider the loop \\$((1)_{x\in X},g\colon \srX\to \{0,1\})$ such that $g(\Id_X)=0$ and $g(\sigma)=1$ for the automorphism $\sigma$ which exchanges two elements in $X$ and thus $g.{\sigma}(\Id_X)=g(\sigma)=1$ and $g.{\sigma}(\sigma)=g(\Id_X)=0$. Then $((1)_{x\in X},g).\sigma\neq ((1)_{x\in X}, g)$ showing $g$ is a nonfixed loop in $L^{\underline A}$.
		
		For example, when $X=\{s,t\}$, the exponential $L^{\underline A}$ has arc set equal to $2^2\x \Set(2^2,2)$, i.e., it has $2^6=64$ elements. Each arc can be represented by a 6-digit binary number. The exponential object $L^{\underline A}$ is given as follows.
		\newline
\end{enumerate}
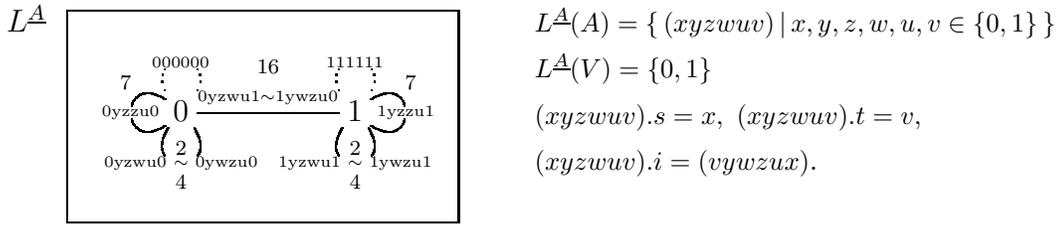
		\begin{multicols}{2}
		\begin{center}
				$L^{\underline A}\ \ $\framebox{ 
					$\xymatrix{\\ \ar@<+.5em>@{}[r]^<<<<{7} & 0 \ar@{}[d]|->>4 \ar@{}@<+1em>[rr]^-{16} \ar@{-}[rr]^{\text{\tiny 0yzwu1$\sim$1ywzu0}} \ar@{..}@(ul,ur)[]|-{\text{\tiny 000000}} \ar@{-}@(ul,dl)[]|-{\text{\tiny 0yzzu0}} \ar@{-}@(dl,dr)[]|-{\text{\tiny 0yzwu0 $\sim$ 0ywzu0}}^2  && 1\ar@{..}@(ul,ur)[]|-{\text{\tiny 111111}} \ar@{-}@(ur,dr)[]|-{\text{\tiny 1yzzu1}}  \ar@{-}@(dl,dr)[]|-{\text{\tiny 1yzwu1 $\sim$ 1ywzu1}}^2 \ar@<.5em>@{}[r]^>>>>7 \ar@{}[d]|->>4 & \\ &  && }$}
			\end{center}
			\begin{align*}
			& \text{\footnotesize $L^{\underline A}(A)=\setm{(xyzwuv)}{x,y,z,w,u,v\in \{0,1\}}$}\\
			& \text{\footnotesize $L^{\underline A}(V)=\{0,1 \}$}\\
			& \text{\footnotesize $(xyzwuv).s=x,\ (xyzwuv).t=v,$}\\ 
			& \text{\footnotesize $(xyzwuv).i=(vywzux)$}.
			\end{align*}
		\end{multicols}
\begin{enumerate}[resume]
\item[]		where $i\colon 2\to 2$ is the non-identity automap. We see that $L^{\underline A}$ has 16 fixed loops (with 7 non-distinguished fixed loops at each vertex), 8 non-fixed loops (4 at each vertex) and 16 edges between vertices.  It is helpful to keep track of the edges associated to the digits $(\underset{s}{x}\underset{\ell_s}{y}\underset{a_1}{z}\underset{a_i}{w}\underset{\ell_t}{u}\underset{t}{v})$. Then evaluation $\ev_A\colon L^{\underline A}(A)\x \underline A\to \{0,1\}$ is given by projection to the corresponding digit, e.g., $\ev_A( (\underset{s}{x}\underset{\ell_s}{y}\underset{a_1}{z}\underset{a_i}{w}\underset{\ell_t}{u}\underset{t}{v}),\ell_s)=y$.\footnote{In \cite{dP}[Proposition 2.3.1], it is proven that the category of conceptual graphs does not have exponentials by attempting to construct the corresponding exponential $L^{\underline A}$. We have given a constructive reason why it failed. Namely, the objects in the category of conceptual graphs lack 2-loops.}
	\end{enumerate}
\end{example}

\section{Interpretations in Conventional Categories}
More conventional categories of graphs, uniform hypergraphs, and hypergraphs can be given as certain comma categories. These are cocomplete categories and thus admit nerve-realization adjunctions between categories of (reflexive) $(X,M)$-graphs induced by the obvious interpretation functors which we define in the subsequent. 

\subsection{Nerve-Realization Adjunction}
Let $I\colon \DD T \to \C M$ be functor from a small category $\DD T$ to a cocomplete category $\C M$. Since the Yoneda embedding $y\colon \C T\to \widehat{\DD T}$ is the free cocompletion of a small category there is a essentially unique adjunction $R\dashv N\colon \C M\to \widehat{\DD T}$, called the nerve realization adjunction, such that $Ry\iso I$. 
\[
\xymatrix{ \DD T \ar[dr]_I \ar[r]^{y} & \widehat{\DD T} \ar@<-.4em>[d]_{R}^{\dashv} \ar@{<-}@<+.5em>[d]^{N}\\ & \C M }
\]
The nerve and realization functors are given on objects by 
\begin{align*}
 N(m) &=\C M(I(-), m),\\
 R(X) &=\colim_{(c,\varphi)\in\int F} I(c)
\end{align*} respectively, where $\int F$ is the category of elements of $X$ (\cite{hA}, Section 2, pp 124-126).\footnote{In \cite{hA}, the nerve functor is called the singular functor.} 

We call a functor $I\colon \DD T \to \C M$ from a small category to a cocomplete category an interpretation functor. The category $\DD T$ is called the theory for $I$ and $\C M$ the modeling category for $I$. An interpretation $I\colon \DD T\to \C M$ is dense, i.e., for each $\C M$-object $m$ is isomorphic to the colimit of the diagram
$I\ls m \to \C M,\ (c,\varphi)\mapsto I(c),$
if and only if the nerve $N\colon \C M\to \widehat{\DD T}$ is full and faithful (\cite{sM}, Section X.6, p 245). When the right adjoint (resp. left adjoint) is full and faithful we call the adjunction reflective (resp. coreflective).\footnote{since it implies $\C M$ is equivalent to a reflective (resp. coreflective) subcategory of $\widehat{\DD T}$}

We are interested in when the nerve also preserves any exponentials which exist. For the purpose of this paper, we show that if an interpretation is dense, full and faithful, then the nerve not only preserves limits, but also any exponentials which exist.

\begin{lemma}\label{L:FFInt}
	An interpretation functor $I\colon \DD T \to \C M$ is full and faithful iff $\underline c\defeq y(c)$ is a $NR$-closed object for each $\DD T$-object $c$, i.e., the unit $\eta_{\underline c}\colon \underline c\to NR(\underline c)$ at component $\underline c$ is an isomorphism.
\end{lemma}
\begin{proof}
	The unit of the adjunction $\eta_G$ is defined as the following composition
	\[
	\xymatrix{ G \ar[r]^-\varphi_-{\iso} & \widehat{\DD T}(y(-), G) \ar[r]^-{R_{(y,G)}} & \C M(Ry(-),R(G)) \ar[r]^-\psi_-{\iso} & \C M(I(-), R(G))=NR(G)}, 
	\]
	where $\varphi$ is given by Yoneda, $R_{(y,G)}$ is the map of homsets given by application of $R$, and $\psi$ is precomposition by the isomorphism $I\iso Ry$. For a representable, $\underline c$, there is an isomorphism $\rho\colon \C M(I(-), R(\underline c))\to \C M(I(-),I(c))$ by postcomposition by the isomorphism $I\iso Ry$. Thus $\rho\circ \psi\circ R_{(y,G)}$ evaluated at $\DD T$-object $c'$ takes a $\DD T$-morphism $f\colon c'\to c$ to $I(f)\colon I(c')\to I(c)$. Thus $I$ is full and faithful iff $\eta_{\underline c}$ is an isomorphism. 
\end{proof}

\begin{proposition}\label{P:Exponential}
	If an interpretation functor $I\colon \DD T\to \C M$ is dense, full and faithful, then  $R\dashv N$ is reflective and $N$ preserves any exponentials that exist in $\C M$. 
\end{proposition}
\begin{proof}
	Suppose $G$ and $H$ are $\C M$-objects such that the exponential $G^H$ exists in $\C M$. Since $I$ is assumed to be full and faithful, by Lemma \ref{L:FFInt} above,  $\underline c\iso NR(\underline c)$ for each $\DD T$-object. Thus we have the following string of natural isomorphism:
	\begin{align*}
	N(G^H)(c) &\iso \C M(R(\underline c)\x H, G) && \text{\footnotesize(Yoneda, $R\dashv N$, exponential adjunction})\\
	&\iso\widehat{\DD T}(NR(\underline c)\x N(H), N(G)) && \text{\footnotesize($N$ is full and faithful, preserves limits)}\\
	&\iso \widehat{\DD T}(\underline c\x N(H), N(G)) &&\text{(\footnotesize$\underline c$ is $NR$-closed)}\\
	&\iso N(G)^{N(H)}(c) && \text{\footnotesize(Exponential adjunction and Yoneda)}.
	\end{align*}
	Since the right-action structures are determined by the Yoneda embedding,  $N(G^H)\iso N(G)^{N(H)}$ in $\widehat{\DD T}$.
\end{proof}

\subsection{Interpretations in Categories of $F$-Graphs}

We follow the definition given in \cite{cJ}. 

\begin{definition}
	Let $F\colon \Set\to \Set$ be an endofunctor. The category of $F$-graphs $\C G_F$ is defined to be the comma category $\C G_F\defeq \Set\ls F$.
\end{definition}
In other words, an $F$-graph $G=(G(E),G(V),\partial_G)$ consists of a set of edges $G(E)$, a set of vertices $G(V)$ and an incidence map $\partial_G\colon G(E)\to F(G(V))$. A morphism \[(f_E,f_V)\colon (G(E),G(V),\partial_G)\to (H(E),H(V),\partial_H)\] is a pair of set maps $f_E\colon G(E) \to H(E)$ and $f_V\colon G(V)\to H(V)$ such that the following square commutes
\[
\xymatrix{ G(E) \ar[d]_{\partial_G} \ar[r]^{f_E} & H(E) \ar[d]^{\partial_H} \\ F(G(V)) \ar[r]^{F(f_V)} & F(H(V)). }
\]
It is well-known that the category of $F$-graphs is cocomplete with the forgetful functor $U\colon \C G_F\to \Set\x \Set$ creating colimits \cite{cJ}.

Let $\DD G_{(X,M)}$ be a theory for $(X,M)$-graphs and $q$ an element in $F(X)$ such that $F(m)(q)=q$ for each $m\in M$ where $m\colon X\to X$ is the right-action map. We define $I(V)\defeq (\empset,\ 1,\ !_1),$ and $I(A)\defeq (1,\ X,\ \named q)$ where $!_1\colon \empset \to 1$ is the initial map and $\named x\colon 1\to X$ the set map with evaluation at $x\in X$. On morphisms, we set
\begin{align*} 
(x\colon V\to A) \quad &\mapsto \quad  I(x)\, \defeq \ (!_1,\named x)\colon (\empset,\ 1,\ !_1)\to (1,\ X,\ \named q),\\
(m\colon A\to A) \quad &\mapsto \quad I(m)\defeq (\Id_1, F(m))\colon (1,X,\named q)\to (1,X,\named q).
\end{align*}
Verification that $I\colon \DD G_{(X,M)}\to \C G_F$ is a well-defined interpretation functor is straightforward. 

\subsection{Interpretations in Reflexive $F$-Graphs}\label{S:Interpr}

For categories of graphs with vertices as degenerate edges, we generalize the definition of conceptual graphs in \cite{dP} (Definition 2.1.1, p 16).

\begin{definition}
	Let $F\colon \Set \to \Set$ be an functor and $\eta\colon \Id_{\Set}\Rightarrow F$ a natural transformation. The category of reflexive $F$-graphs $\crG_F$
	has objects $G=(G(P),G(V),\partial_G)$ where $G(P)$ is a set, $G(V)\subseteq G(P)$ is a subset and $\partial_G\colon G(P)\to F(G(V))$ is a set map. An $F$-graph morphism $f\colon G\to H$ consists of a set map $f_P\colon G(P)\to H(P)$ such that the following commutes
	\[
	\xymatrix@R=1.5em@C=1.5em{ &  G(V) \ar@{>->}[dl] \ar[dd]|-\hole^<<<<{\eta}  \ar[r]^{f_V} & H(V) \ar[dd]|-\hole_<<<<{\eta}  \ar@{>->}[dr] & \\ G(P) \ar[dr]_{\partial_G} \ar[rrr]^{f_P} &&& H(P) \ar[dl]^{\partial_H} \\ & F(G(V)) \ar[r]^{F(f_V)} & F(H(V))}
	\]
	where $f_V$ is the set map $f_P$ restricted to $G(V)$.\footnote{By naturality $\eta\colon \Id_{\Set}\Rightarrow F$ the middle square always commutes.}
\end{definition}
In other words, a reflexive $F$-graph $G$ consists of parts $G(P)$ with a subset of vertices $G(V)$ and an incidence operation $\partial_G\colon G(P)\to F(G(V))$ which considers a vertex $v$ to be a degenerate edge in the sense that $\partial_{G|G(V)}=\eta$. A reflexive $F$-graph morphism $f\colon G\to H$ that maps an edge to a vertex is one where $e\in G(P)\backslash G(V)$ has $f_P(e)\in H(V)$. 

The category of reflexive $F$-graphs is cocomplete. Indeed, the empty $F$-graph is the initial object. Given a family of $F$-graphs $(G_i)_{i\in I}$ the coproduct is given by taking the disjoint union of parts with incidence operator induced by the universal property of the coproduct on the cocone \[(\xymatrix{G_i(P)\ar[r]^-{\partial_{G_i}} &   F(G_i(V)) \ar[r]^-{F(s_i)} & F(\bigsqcup_I G_i(V)) })_{i\in I}\] where $s_i\colon G_i(V)\to \bigsqcup_I G_i(V)$ is the coproduct inclusion. Given a pair of morphisms $f,g\colon G\to H$, the coequalizer $\coeq(f,g)$ has part set equal to $H(P)/\sim$ where $\sim$ is the equivalence generated by the relation 
$f(a)\sim g(a)$ for each $a\in G(P)$
 and vertex set equal to the image of $H(V)\to H(P)/\sim$. The incidence $\partial_{\coeq(f,g)}\colon \coeq(f,g)\to F(\coeq(f,g)(V))$ is induced by the universal property of coequalizer. 
\[
\xymatrix{ G(P) \ar@<+.3em>[r]^f \ar@<-.3em>[r]_g & H(P) \ar[d]_{\partial_H} \ar[r] & \coeq(f,g) \ar@{..>}[dl] \ar[d]^{\partial_{\coeq(f,g)}}  \\ & F(H(V)) \ar[r] & F(\coeq(f,g)(V))}
\]
It is straightforward to verify these are well-defined reflexive $F$-graphs which enjoy universal properties.

Let $\rG_{(X,M)}$ be a theory for reflexive $(X,M)$-graphs. Define the set $M_A\defeq \frac{M}{\sim}$ where $\sim$ is the equivalence relation such that $m\sim m'$ iff there exists an invertible $n\in M$ such that $mn=m'$. This makes $M_A$ a right $M$-set with the obvious action. Let $q\colon M_A\to F(X)$ be a set map such that for each $m\in M$ we have $F(m)\circ q=q\circ m$ 
\[
	\xymatrix{ M_A \ar[r]^{m} \ar[d]_q & M_A \ar[d]^{q} \\ F(X) \ar[r]^{F(m)} & F(X) }
\]
where $m\colon M_A\to M_A$ is the right-action map.
 Define $I(V)(P)=1$ (and thus has a single vertex with no edges) and $I(A)(P)=M_A$, with vertex set $I(A)(V)=X$ and inclusion $I(A)(V)\hookrightarrow I(A)(P)$\footnote{Recall $X=\Fix(M)$.} with incidence defined by $\delta_{I(A)}\defeq q$. For morphisms we assign for each $x_m\in X$ and $m\in M$ 
\begin{align*}
(x_{m'}\colon V\to A)\quad \, &\mapsto \quad  I(x_{m'})_P\defeq \named{m'}\colon 1\to M_A,\\
(m\colon A\to A)\quad &\mapsto \quad  I(m)_P\defeq m\colon M_A\to M_A ,\\
(\ell\colon A\to V) \quad \, &\mapsto \quad  I(\ell)_P\defeq !_{M_A}\colon M_A\to 1 \ \text{(the terminal set map)}
\end{align*}
which is readily verified to define an interpretation functor $I\colon \rG_{(X,M)}\to \crG_F$.

In the following, we will consider the properties of the nerve realization adjunction $R\dashv N$ induced by $I$ as well as the restriction to an adjoint equivalence between fixed points.\footnote{Recall that the fixed points of an adjunction $F\dashv G\colon \C A\to \C B$ are the full subcategories $\C A'$ and $\C B'$ of $\C A$ and $\C B$ consisting of objects such that the counit and unit of the adjunction are isomorphisms. This in particular implies that $\C A'$ is equivalent to $\C B'$. } 

\subsection{The Category of Hypergraphs}\label{S:PGraphs}

We recall that a hypergraph $H=(H(V),H(E),\varphi)$ consists of a set of vertices $H(V)$, a set of edges $H(E)$ and an incidence map $\varphi\colon H(E)\to \C P(H(V))$ where $\C P\colon \Set\to \Set$ is the covariant power-set functor. In other words, we allow infinite vertex and edge sets, multiple edges, loops, empty edges and empty vertices.\footnote{An empty vertex is a vertex  not incident to any edge in $H(E)$. An empty edge is an edge $e$ such that $\varphi(e)=\empset$.} In other words the category of hypergraphs $\C H$ is the category of $\C P$-graphs.

Let $X$ be a set and apply the definition for the interpretation given above in \ref{S:Interp} for $\sG_X$ with $q\defeq X$ in $\C P(X)$. Note that for each automap $\sigma\colon X\to X$, $\C P(\sigma)$ is the identity map. Thus the interpretation $I\colon \sG_X\to \C H$ defined in Chapter \ref{S:Interp} is a well-defined functor. 

The nerve $N\colon \C H\to \widehat{\sG}_{X}$ induced by $I$ takes a hypergraph $H=(H(E),H(V),\varphi)$ to the symmetric $X$-graph $N(H)$ with vertex and arc set given by
\begin{align*}
N(H)(V)&=\C H(I(V),H)=H(V),\\
N(H)(A)&=\C H(I(A),H)=\setm{(\beta,f)\in H(E)\x H(V)^X}{\C P(f)=\varphi(\beta)}
\end{align*} Notice that in the case a hyperedge $e$ has less than $\#X$ incidence vertices the nerve creates multiple edges and if a hyperedge has more than $\#X$ incidence vertices there is no arc in the correponding symmetric $X$-graph given by the nerve.

The realization $R\colon \widehat{\sG}_{X}\to \C H$ sends a symmetric $X$-graph $G$ to the hypergraph $R(G)=(R(G)(E),R(G)(V),\psi)$ with vertex, edge sets and incidence map given by
\begin{align*}
& R(G)(V)=G(V),\\  & R(G)(E)=G(A)/\sim, \quad \text{($\sim$ induced by $\sX$)},\\
&\psi\colon R(G)(E)\to \C P(R(G)(V)), \quad [\gamma]\mapsto \setm{v\in G(V)}{\exists x\in X, \gamma.x=v}
\end{align*} 
For a symmetric $X$-graph morphism $f\colon G\to G'$, the hypergraph morphism $R(f)\colon R(G)\to R(G')$ has $R(f)_V\defeq f_V$ and $R(f)_E\defeq [f_A]$ where $[f_A]\colon \frac{G(A)}{\sim} \to \frac{G'(A)}{\sim}$ is induced by the quotient. 

Let $k$ be a cardinal number. Recall that a hypergraph $H=(H(E),H(V),\varphi)$ is $k$-uniform provided for each edge $e\in H(E)$, the set $\varphi(e)$ has cardinality $k$.

\begin{proposition} \label{P:HyperGSG}
	Let $k$ be the cardinality of $X$ and $I\colon \sG_X\to \C H$ be the interpretation above. The fixed points of the nerve realization adjunction $R\dashv N\colon \C H\to \sG_X$ is equivalent to the category of $k$-uniform hypergraphs, $k\C H$.  Moreover, the inclusion $i\colon k\C H\to \widehat{\sG}_X$ preserves limits and any exponential objects which exist in $k\C H$. 
\end{proposition}
\begin{proof}
	It is clear that the fixed points is the category of $k$-uniform hypergraphs and that the product (respectively, equalizer) of $k$-uniform hypergraphs in $\widehat{\sG}_X$ is $k$-uniform. Thus the inclusion $i\colon k\C H\to \widehat{\sG}_X$ preserves limits. To show that $N$ must preserve any exponentials that exist, suppose $G^H$ is an exponential object in $k\C H$. We have the following natural isomorphisms:
	\begin{align*}
		N(G^H)(V)&=k\C H(I(V),G^H)\cong k\C H(I(V)\x H, G)\\
		 &\cong \widehat{\sG}_X(NI(V)\x N(H),N(G))\\
		 &\cong \widehat{\sG}_X(\underline V,N(G)^{N(H)})\cong N(G)^{N(H)}(V),\\\\
		 N(G^H)(A)&=k\C H(I(A),G^H)\cong k\C H(I(A)\x H, G)\\
		 &\cong \widehat{\sG}_X(NI(A)\x N(H),N(G))\\
		 &\cong \widehat{\sG}_X(\underline A,N(G)^{N(H)})\cong N(G)^{N(H)}(A).
	\end{align*}
Therefore, $N$ preserves any exponentials which exist in $k\C H$. 
\end{proof}
\begin{corollary}
If $k$ is a cardinal number greater than $1$, the category of $k$-uniform hypergraphs does not have exponentials. 	
\end{corollary}
\begin{proof}
	Example \ref{E:Exponentials}(\ref{E:HypergraphFail}) provides us with a counterexample. 
\end{proof}

\subsection{The Category of Power Graphs} \label{S:UnderlinePiGraphs}

Let $X$ and $Y$ be sets. We define the symmetric $X$-power of $Y$, denoted $\underline \Pi_X(Y)$, as the multiple coequalizer of $(\underline \sigma\colon \Pi_X(Y) \to \Pi_X(Y))_{\sigma\in \sX}$ where $\underline \sigma$ is the $\sigma$-shuffle of coordinates in the product. This definition extends to a functor $\underline{\Pi} _{X}\colon \Set \to \Set$. Note that if $j\colon X'\to X$ is a set map, then there is a natural transformation $\underline{\Pi} _{X}\Rightarrow \underline{\Pi} _{X'}$ induced by the universal mapping property of the product. In particular, when $X\to X'=1$ is the terminal map, we have $\Id_{\Set}=\underline{\Pi} _1\Rightarrow \underline{\Pi} _X$ which we denote by $\eta\colon \Id_\Set\Rightarrow \underline{\Pi} _X$.\footnote{Note that in the case $X=2$, the category of $\underline \Pi_X$-graphs is the category of undirected graphs in the conventional sense in which morphisms are required to map edges to edges.  }

To define an interpretation functor $I\colon \sG_{X}\to \C G_{\underline \Pi_X}$, we let $q$ be the unordered set $(x)_{x\in X}$ in $\underline \Pi_X(X)$.  Since $\underline \Pi_X(\sigma)(x)_{x\in X}=(x)_{x\in X}$ for each automap $\sigma\colon X\to X$, the interpretation is well-defined.

\begin{lemma}
	The interpretation $I\colon \sG_{X}\to \C G_{\underline \Pi_X}$ is dense, full and faithful.
\end{lemma}
\begin{proof}
	It is clearly full and faithful. To show it is dense, let $(E,V,\varphi)$ and $(K,L,\psi)$ be $\C G_{\underline \Pi_X}$-objects and $\lambda\colon D\Rightarrow \Delta(K,L,\psi)$ a cocone on the diagram $D\colon I\ls (E,V,\varphi)\to \C G_{\underline \Pi_X}$. 
	Let $e$ be an edge in $E$ and $f\colon X\to V$ be the set morphism with $\underline{\Pi} _Xf=\varphi(e)$. Then $(\named e, f)\colon I(A)=(1,X,\top)\to (E,V,\varphi)$ is an object in $I\ls (E,V,\varphi)$ and thus there is a morphism $\lambda_{(\named {e},f)}\eqdef(\named{e'},g)\colon D(\named e,f)=(1,X,\top)\to (K,L,\psi)$. By the compatibility of the cocone, this gives us a uniquely defined $h\colon E\to K$, $e\mapsto e'$ on edges. Similarly for each vertex $v\in V$, there is a morphism $(!_E,\named v)\colon I(V)=(\empset, 1,!_1)\to (E,V,\varphi)$ and a cocone inclusion $(!_K,\named w)\colon D(!_E,\named v)=(\empset, 1,!_1)\to (K,L,\psi)$ giving us a factorization on vertices $k\colon V\to L$. Since $\psi\circ h(e)=\underline{\Pi} _X(kf)\circ \top=\underline{\Pi} _X(k)\circ \varphi(e)$ for each edge $E$,  $(h,k)\colon (E,V,\varphi)\to (K,L,\psi)$ is a well-defined $\C G_{\underline \Pi_X}$-morphism which necessarily is the unique factorization of the cocone. Therefore, $I$ is dense. 
\end{proof}

Note that the realization functor takes a $\widehat{\sG}_{X}$-object and quotients out the set of arcs by $\sX$. Hence the unit of the adjunction $\eta_P\colon P\to NR(P)$ is bijective on vertices and surjective on arcs. Hence the adjunction is epi-reflective. 

For a $\C G_{\underline \Pi_X}$-object $(B,C,\varphi)$, the embedding given by the nerve functor is given by 
\begin{align*}
N(B,C,\varphi)(V)&=\C G_{\underline \Pi_X}(I(V), (B,C,\varphi))\iso C,\\
N(B,C,\varphi)(A)&=\C G_{\underline \Pi_X}(I(A), (B,C,\varphi))\\ 
&=\setm{(e,g)\ }{\ e\in B,\ g\colon X\to C\ s.t.\ \underline{\Pi} _Xg=\varphi(e)} 
\end{align*}
The  right-actions are by precomposition, i.e., $(e,g).x=(e,g\circ \named x)$, $(e,g).\sigma=(e,g\circ \sigma)$. 

Let us show that all loops in the objects of the full subcategory of $\widehat{\sG}_{X}$ equivalent to $\C G_{\underline \Pi_X}$ are fixed loops. A loop in a $\C G_{\underline \Pi_X}$-object $(B,C,\varphi)$ is an edge $e\in B$ such that $\varphi(e)$ is $(v)_{x\in X}$ in $\underline{\Pi} _X(C)$ for some $v\in C$.  Therefore, there is only one morphism $(\named e, f)\colon I(A)\to (B,C,\varphi)$ and thus $(\named e,f\circ \sigma)=(\named e,f)$ for each $\sigma\in \sX$. Hence, each object in the reflective subcategory of $\widehat{\sG}_{X}$ equivalent to $\C G_{\underline \Pi_X }$ has only fixed loops.

\begin{corollary}\label{C:Exponentials}
	If $X$ has cardinality greater than $1$, the category $\C G_{\underline \Pi_X}$ does not have exponentials. 
\end{corollary}
\begin{proof}
	By the above observation, it is enough to show that there exist objects $G$ and $H$ in $\C G_{\underline \Pi_X}$ such that $N(G)^{N(H)}$ has a nonfixed loop in $\widehat{\DD G}_{(X,\sX)}$. Set $H\defeq I(A)$ and $G$ be the graph with one vertex and an $\sX$-loop. Then $N(G)^{N(H)}=L^{\underline A}$ as defined in Example \ref{E:Exponentials}(\ref{E:XLoops}) which we have shown has a nonfixed loop. 
\end{proof}

\subsection{The Category of Reflexive Power Graphs}

Let $r\C G_{\underline \Pi_X}$ be the category of reflexive $\underline \Pi_X$-graphs.\footnote{When $X=2$, the category of reflexive $\underline \Pi_X$-graphs is the category of conceptual graphs as given in \cite{dP}.}
To define an interpretation functor $I\colon \srG_{X}\to r\C G_{\underline \Pi_X}$, note that $M_A\iso X\sqcup 1$. Let $\eta\colon \Id_\Set\Rightarrow \underline \Pi_X$ be the natural transformation defined above and let $q\colon X\sqcup 1\to \underline \Pi_X(X)$ the map induced by the singleton assignment $\eta_X\colon X\to \underline \Pi_X(X)$, $x'\mapsto (x')_{x\in X}$ and $\top\colon 1\to \C P(X)$, $x\mapsto (x)_{x\in X}$.  Since $\underline \Pi_X(\sigma)(x)_{x\in X}=(x)_{x\in X}$ for each automap $\sigma\colon X\to X$ and $\underline \Pi_X(\overline{x'})(x)=(x')_{x\in X}$ for each constant map $\overline {x'}\colon X\to X$, the interpretation is well-defined. 

\begin{lemma}
	The interpretation functor $I\colon \srG_{X}\to r\C G_{\underline \Pi_X}$ is dense, full and faithful.
\end{lemma}
\begin{proof}
	It is clearly full and faithful. To show it is dense, let $G$ and $H$ be $r\C G_{\underline \Pi_X}$-objects and $\lambda\colon D\Rightarrow \Delta H$ a cocone on the canonical diagram $D\colon I\ls G \to r\C G_{\underline \Pi_X}$. It can be verified that $I(A)$ classifies the parts set $G(P)$ of a graph $G$ up to precomposition by automorphism $A'\to A'$. In other words, $G(P)\iso \frac{r\C G_{\underline \Pi_X}(I(A),G)}{\sim}$ and $H(P)\iso \frac{r\C G_{\underline \Pi_X}(I(A),H)}{\sim}$ where $\sim$ is the equivalence relation induced by automorphisms of $I(A)$. Thus we define $h_P\colon G(P)\to H(P)$, $[e]\mapsto [\lambda_{e}]$ where $[e]$ is the equivalence class of the morphism $e\colon I(A)\to G$ and $\lambda_e\colon D(e)\to H$ is the component of the natural transformation $\lambda$. Since $\lambda$ is a cocone, the map is compatible with incidence operations and the restriction to vertex sets, $h_V\colon G(V)\to H(V)$. Thus $h\colon G\to H$ is the unique factorization which shows the colimit of $D$ is $G$.
\end{proof}

Note that the realization functor takes a $\widehat{\srG}_{X}$-object and quotients out the set of arcs by $\sX$. Hence the unit of the adjunction $\eta_P\colon P\to NR(P)$ is bijective on vertices and surjective on arcs. Hence the adjunction is epi-reflective. 

The full subcategory of $\widehat{\srG}_{X}$ induced by the nerve functor consists of reflexive symmetric $X$-graphs which have no nonfixed loops. Indeed if $G$ is a $r\C G_{\underline \Pi_X}$-object then $N(G)(A)=r\C G_{\underline \Pi_X}(I(A),G)$ and so if $e\colon I(A)\to G$ is a loop, i.e., for each $x\in X$ there is a $v\colon I(V)\to I(A)$ such that $e\circ I(x)=v$, then $e\circ I(\sigma)=e$. 
\begin{corollary}\label{C:NoExpR}
	If $X$ has cardinality greater than $1$, the category $r\C G_{\underline \Pi_X}$ does not have exponentials.  
\end{corollary}
\begin{proof}
	By the above observation, it is enough to show that there exist objects $G$ and $H$ in $\C G_{\underline \Pi_X}$ such that $N(G)^{N(H)}$ has a 1-loop in $\widehat{\DD G}_{(X,\sX)}$. Set $H\defeq I(A)$ and $G$ be the graph with one vertex and two  1-loops. Then $N(G)^{N(H)}=L^{\underline A}$ as defined in Example \ref{E:Exponentials}(\ref{E:1Loops}) which we have shown has a nonfixed loop.
\end{proof}

\end{document}